\RequirePackage{rotating}
\documentclass[12pt]{amsart}
\usepackage[graphicx]{realboxes}
\usepackage{float}
\restylefloat{table}
\usepackage{placeins}

\usepackage[T1]{fontenc}
\usepackage[utf8]{inputenc}
\usepackage{enumitem}
\usepackage{amsmath}
\usepackage{amssymb}
\usepackage{epsfig}
\usepackage{wasysym}
\usepackage{graphicx}
\usepackage{tikz}
\usepackage{tikz-cd}
\usepackage{bm}
\usepackage{xcolor}
\usepackage{listings}
\numberwithin{equation}{section}
\usepackage{extpfeil}
\usepackage{array}
\usepackage{adjustbox}
\usepackage{longtable}
\usepackage{makecell}
\usepackage[all]{xy}
\usepackage{longtable}
\usepackage{rotating}
\input xy
\xyoption{all}
\usepackage{multirow}

\usetikzlibrary{positioning}
\usepackage[colorlinks=false,urlbordercolor=white]{hyperref}
  
\tikzset{sgplattice/.style={inner sep=1pt,norm/.style={red!50!blue},char/.style={blue!50!black},
  lin/.style={black!50}},cnj/.style={black!50,yshift=-2.5pt,left=-1pt of #1,scale=0.5,fill=white}}

\DeclareFontFamily{U}{mathb}{\hyphenchar\font45}
\DeclareFontShape{U}{mathb}{m}{n}{
      <5> <6> <7> <8> <9> <10> gen * mathb
      <10.95> mathb10 <12> <14.4> <17.28> <20.74> <24.88> mathb12
      }{}
\DeclareSymbolFont{mathb}{U}{mathb}{m}{n}
\DeclareMathSymbol{\righttoleftarrow}{3}{mathb}{"FD}

\calclayout
\allowdisplaybreaks[3]

\theoremstyle{plain}
\newtheorem{prop}{Proposition}[section]

\newtheorem{theo}[prop]{Theorem}
\newtheorem{coro}[prop]{Corollary}

\newtheorem{lemm}[prop]{Lemma}

\theoremstyle{definition}

\newtheorem*{prob*}{Problem}

\newtheorem{rema}[prop]{Remark}

\newtheorem{exam}[prop]{Example}

\newcommand{\actsfromleft}{\mathrel{\reflectbox{$\righttoleftarrow$}}}
\newcommand{\actsfromright}{\righttoleftarrow}

\def\lra{\longrightarrow}

\def\rk{\mathrm{rk}}

\def\cA{{\mathcal A}}

\def\fA{{\mathfrak A}}

\def\fD{{\mathfrak D}}

\def\fS{{\mathfrak S}}

\def\fS{{\mathfrak S}}

\def\bP{{\mathbb P}}

\def\bZ{{\mathbb Z}}

\def\bF{{\mathbb F}}

\def\Pic{\mathrm{Pic}}

\def\Aut{\mathrm{Aut}}
\def\ZAut{\mathrm{ZAut}}

\def\SL{\mathsf{SL}}
\def\PSL{\mathsf{PSL}}
\def\GL{\mathsf{GL}}

\def\PGL{\mathsf{PGL}}

\def\Out{\mathrm{Out}}

\def\Burn{\mathrm{Burn}}
\def\Bir{\mathrm{Bir}}
\def\ZBir{\mathrm{ZBir}}

\def\lim{\mathrm{lim}}

\def\Cr{\mathrm{Cr}}

\makeatother
\makeatletter

\begin{document}
\title[Conjugacy in the plane Cremona group]{Conjugacy classes of linear actions\\ in the plane Cremona group}

\author[I. Cheltsov]{Ivan Cheltsov}
\address{Department of Mathematics, University of Edinburgh, UK}
\email{I.Cheltsov@ed.ac.uk}

\author[Y. Tschinkel]{Yuri Tschinkel}
\address{
  Courant Institute,
  251 Mercer Street,
  New York, NY 10012, USA
}

\email{tschinkel@cims.nyu.edu}

\address{Simons Foundation\\
160 Fifth Avenue\\
New York, NY 10010\\
USA}

\author[Zh. Zhang]{Zhijia Zhang}

\address{
Courant Institute,
  251 Mercer Street,
  New York, NY 10012, USA
}

\email{zhijia.zhang@cims.nyu.edu}

\date{\today}

\begin{abstract}
We classify regular generically free actions of finite groups on the projective plane, up to conjugation in the Cremona group. 
\end{abstract}

\maketitle

\section{Introduction}
\label{sect:intro}

One of the most intensely studied objects in algebraic geometry is the plane Cremona group 
$$
\Cr_2=\Bir(\bP^2),
$$
the group of birational automorphisms of the plane.
Recall that, over an algebraically closed field of characteristic zero,
$$
\Cr_2:=\langle \PGL_3, \iota\rangle,
$$
where
\begin{equation} 
\label{eqn:crem}
\begin{array}{rccc} 
\iota:&  \bP^2 & \dashrightarrow & \bP^2, \\
& (x,y,z) & \mapsto & \left(\frac{1}{x},\frac{1}{y},\frac{1}{z}\right)
\end{array}
\end{equation}
is the standard Cremona involution.
Conjugacy classes of finite subgroups of $\PGL_3$ have been described in \cite{blichfeldt}: there are three types, recalled in Section \ref{sect:plane}:
\begin{itemize}
    \item intransitive, 
    \item transitive but imprimitive, 
    \item primitive. 
\end{itemize}
There exist isomorphic non-conjugated {\em finite subgroups} of $\PGL_3$ which are conjugated in $\Cr_2$, see Sections~\ref{sect:plane} and \ref{sect:intran}. However, finite subgroups of different types are not conjugated in $\Cr_2$. 

A century later, in the seminal paper by Dolgachev and Iskovskikh \cite{DI},
came the classification of 
finite subgroups of $\Cr_2$, i.e., finite groups
that can act regularly and generically freely on rational surfaces.  However, the classification
of {\em subgroups} as well as of {\em actions}, 
up to conjugation in $\Cr_2$, remained an open problem.
For finite subgroups of $\PGL_3$, this was explicitly asked in \cite[Section 9]{DI}: 
\begin{quote}
{\em Find the conjugacy classes in $\Cr_2$ of actions of finite subgroups of $\PGL_3$. For example, there are two actions of the subgroup of $\PGL_3$ isomorphic to $\fA_5$ and four actions of $\fA_6$ which differ by outer automorphisms of the groups. Are they conjugated in $\Cr_2$?}
\end{quote}
Since then, it was shown in  \cite{two-local} that the two $\fA_5$-actions are conjugated in $\Cr_2$, while the four $\fA_6$-actions are not. A similar problem for the remaining primitive actions follows from \cite{sako}, see Section~\ref{section:primitive} for an explicit solution. 
Furthermore, conjugacy classes of actions of intransitive {\em abelian} subgroups of $\PGL_3$ have been described in \cite{reichsteinyoussininvariant}; in particular, actions of finite cyclic subgroups are always conjugated in $\Cr_2$. 

In this paper, we completely settle the Dolgachev--Iskovskikh problem for the remaining finite subgroups of $\PGL_3$. Our main new tool to distinguish the actions up to conjugation in $\Cr_2$ is the formalism of Burnside invariants, introduced in \cite{BnG} and recalled in Section~\ref{sect:gen}. These fail to distinguish primitive actions \cite[Proposition 7.1 and Example 7.2]{TYZ-3}, and are inconclusive for some imprimitive actions \cite{TYZ-3}, but those are accessible via the equivariant Sarkisov program. However, the new invariants are decisive in the case of intransitive actions. One of the main results of this paper, proved in Section~\ref{sect:intran}, is:

\begin{theo}
\label{thm:intran}
Nonabelian intransitive actions $\phi_1,\phi_2: G\hookrightarrow \PGL_3$ are conjugated in $\Cr_2$ if and only if
$$
[\bP^2\actsfromright \phi_1(G)] =
[\bP^2\actsfromright \phi_2(G)] \in \Burn_2^{\mathrm{inc}}(G).
$$
\end{theo}
We also obtain an explicit geometric description of this condition, see Theorem~\ref{thm:intransBurn}.

Our second main result, proved in Section~\ref{sect:trans}, is:

\begin{theo}
\label{thm:imprim}
Distinct imprimitive actions $\phi_1,\phi_2: G\hookrightarrow \PGL_3$ are conjugated in $\Cr_2$ if and only if they are conjugated by the standard Cremona involution.
\end{theo}

Together with existing results concerning primitive actions, this settles the Dolgachev--Iskovskikh problem, quoted above \cite[Section 9]{DI}. 
The proof proceeds via classification of groups and actions, followed by a detailed analysis of equivariant Sarkisov links. 
As a byproduct, we compute the normalizer
of $G$ in the Cremona group in many cases, and answer a question posed by L. Katzarkov:

\begin{theo}
\label{thm:P2-norm}
Let $G\subset \PGL_3$ be a finite subgroup.
Then the normalizer of $G$ in $\Cr_2$
is finite if and only if $G$ is transitive
and not isomorphic to any of the following groups: 
$$
\mathfrak{A}_4,\quad  C_3^2, \quad C_3\rtimes\mathfrak{S}_3, \quad  C_7\rtimes C_3, \quad  C_3^2\rtimes C_4.
$$
\end{theo}

Recall that $\Cr_2$ is generated by involutions, over any perfect field \cite{LS24}. Thus, it is reasonable to expect that the normalizer in $\Cr_2$ of any finite subgroup $G\subset \Cr_2$ is also generated by involutions. However, this is not always the case: if $G\subset \PGL_3$ is isomorphic to the {\em Hessian group}, then $G$ is unique in $\Cr_2$, up to conjugation \cite{DI}, coincides with its normalizer in $\Cr_2$ by \cite{sako}, but is not generated by reflections. One can check that normalizers of other transitive subgroups in $\PGL_3$ are generated by reflection. 

Keeping in mind results in \cite{Serge,Anne,Susanna}, one could ask: When is the normalizer in $\Cr_2$ of a finite subgroup $G\subset \Cr_2$ simple? Such $G$ would have to be simple and coincide with their normalizer. By \cite{two-local}, there are exactly two such nonabelian finite subgroups in $\Cr_2$ (up to conjugation), both contained in $\PGL_3$, namely, $\mathfrak{A}_6$ and $\mathrm{PSL}_2(\mathbb{F}_7)$. Using the classification of elements of prime order in $\Cr_2$, completed in \cite{BayleBeauville,deFernex2004,BeauvilleBlanc,Blanc2}, one can also show that a prime order cyclic subgroup $G\subset\Cr_2$ coincides with its normalizer if and only if $G$ is generated by a Geiser or Bertini involution, and the fixed curve of the involution does not admit nontrivial automorphisms.
Such subgroups are conjugated in $\Cr_2$ if and only if the fixed curves of their generating involutions are isomorphic.

\

\noindent
{\bf Acknowledgments:} 
Ivan Cheltsov has been supported by the Leverhulme Trust grant RPG-2021-229,
EPSRC grant EP/Y033485/1, and Simons Collaboration grant \emph{Moduli of varieties}.
Yuri Tschinkel has been supported by NSF grant 2301983.
This work started during the workshop ``Birational Geometry and Number Theory'' held at ICMS (Edinburgh) in November 2024, has been developed at CIRM (Luminy) during of a semester-long \textit{Morlet Chair} program, and was finished at the Simons Foundation (New York) in July 2025.

\section{Basic notions in equivariant geometry} 
\label{sect:gen}

\subsection*{Automorphisms and birational automorphisms}
Throughout, $k$ is an algebraically closed field of characteristic zero. 
Let $X$ be a smooth projective irreducible variety over $k$ and  
$$
\Aut(X)\subseteq \Bir(X),
$$
the group of regular, respectively, birational, automorphisms of $X$.
Let $G$ be a finite group admitting an embedding
$$
\phi: G \hookrightarrow \Aut(X),
$$
we refer to this as a choice of a $G$-action on $X$, and call $X$ a $G$-variety.   
For $G\subseteq \Aut(X)$, let
$$ 
\Aut^G(X)\subseteq \Bir^G(X),
$$
be the normalizers of $G$ in the respective groups. This group-theoretic definition has the following geometric interpretation: $\Aut^G(X)$ is the group of all $G$-{\em biregular} self-maps of $X$, i.e., all all $\varphi\in \Aut(X)$ such that there exists a 
$\psi\in \Aut(G)$ and a commutative diagram

\ 

\centerline{
\xymatrix{
X \ar[d]_{g} \ar[rr]^{\varphi} & & X\ar[d]^{\psi(g)}\\
X \ar[rr]^{\varphi}                     & &  X   
}
}

\ 

\noindent 
for all $g\in G$. Similarly, $\Bir^G(X)$ is the group of all $G$-{\em birational} 
self-maps, defined by the same diagram, with $\varphi\in \Bir(X)$. There are natural homomorphisms

\
\begin{equation}
\label{eqn:diag}
\xymatrix{
\Aut^G(X) \ar[dr]_{\alpha} \ar@{^{(}->}[rr]  &  &  \Bir^G(X)\ar[dl]^{\beta} \\
 & \Aut(G) &}
\end{equation}
\

\noindent
and we denote the kernels, i.e., centralizers of the respective groups, by 
$$
\ZAut^G(X) \subseteq \ZBir^G(X).  
$$

A birational map $X\dashrightarrow Y$ between $G$-varieties is called 
\begin{itemize}
    \item $G$-equivariant, if it commutes with the action of $G$, 
    \item $G$-birational, if it commutes with the action of $G$, up to a fixed automorphism $\psi\in \Aut(G)$.
\end{itemize}

\subsection*{Conjugacy of groups}
Given a $G$-variety $X$ and embeddings 
$$
\phi_1, \phi_2:  G\hookrightarrow \Aut(X)
$$ 
one has the following general

\

\noindent
{\bf Problem A}:
    When are $\phi_1(G)$ and $\phi_2(G)$ 
    conjugated in $\Aut(X)$?
 
    \
    
 \noindent
{\bf Problem B}:
    When are $\phi_1(G)$ and $\phi_2(G)$ conjugated 
    in  $\Bir(X)$?

\

In the first case, this means that there is a $G$-biregular self-map $\varphi: X\to X$ conjugating $\phi_1(G)$ and $\phi_2(G)$. 
In the second case, $\varphi$ is $G$-birational.

\subsection*{Conjugacy of actions}
Consider actions 
$\phi_1, \phi_2:  G\hookrightarrow \Aut(X)$ on a $G$-variety $X$ as above. They are conjugated in $\Aut(X)$ if there exists a $\varphi\in \Aut(X)$ such that 
$$
\phi_2=\varphi\circ \phi_1\circ \varphi^{-1}. 
$$
In this case, we don't distinguish these two actions! In the following, when we refer to action, we always mean up to conjugation in $\Aut(X)$.  

Clearly, if $\phi_1$ and $\phi_2$ give the same action on $X$, then the groups $\phi_1(G)$ and $\phi_2(G)$ are conjugated in $\Aut(X)$.
However, the converse does not always hold. 

\ 

\noindent 
{\bf Problem AA}:
Assuming that the groups $\phi_1(G)$ and $\phi_2(G)$
are conjugated in $\Aut(X)$, when are the actions $\phi_1,\phi_2$ the same?

\ 

We can consider a similar problem in $\Bir(X)$: we say that the actions $\phi_1$ and $\phi_2$ are conjugated in $\Bir(X)$ if there exists a 
$\varphi\in \Bir(X)$ such that 
$$
\phi_2=\varphi\circ \phi_1\circ \varphi^{-1}. 
$$
As before, if $\phi_1$ and $\phi_2$ are conjugated in $\Bir(X)$, then the groups $\phi_1(G)$ and $\phi_2(G)$ are conjugated in $\Bir(X)$.

\

\noindent 
{\bf Problem BA}:
Assuming that the groups $\phi_1(G)$ and $\phi_2(G)$
are conjugated in $\Bir(X)$, when are the actions $\phi_1,\phi_2$ conjugated in $\Bir(X)$?

\

To solve these problems, we can precompose one of the actions with a suitable $G$-biregular self-map $\varphi\in \Aut(X)$, respectively, with a $G$-birational map 
$\varphi\in \Bir(X)$,
to obtain  
$$
\phi_1(G)=\phi_2(G),  
$$
as subgroups of $\Aut(X)$.   
Then $\phi_2=\phi_1\circ \psi$, for some $\psi \in \Aut(G)$, and the actions are conjugated in $\Aut(X)$ if and only if $\psi$ lies in the image of 
the homomorphism $\alpha$ from diagram \eqref{eqn:diag}.  
We have a diagram
\[
\xymatrix{
\phi_1(G)=\phi_2(G)  \ar@{^{(}->}[r]  & \Aut^G(X) \ar[d]_{\alpha} \ar@{^{(}->}[r]  &  \Bir^G(X) \ar[dl]^{\beta}\\
\mathrm{Inn}(G)  \ar@{^{(}->}[r]  & \Aut(G) \ar@{->>}[r]  &  \mathrm{Out}(G),}
\]
\ 

\noindent
where 
$\mathrm{Inn}(G)$ and  $\mathrm{Out}(G)$ are inner and outer automorphisms of $G$. 
Clearly, if $\psi\in \mathrm{Inn}(G)$, the actions are conjugated in $\Aut(G)$. 
Therefore, Problem AA reduces to

\ 

\noindent
{\bf Problem IA}:
Determine the image of 
$$
\bar{\alpha}: \Aut^G(X) \to \mathrm{Out}(G). 
$$

 \ 

\noindent 
Similarly, Problem BA reduces to

\ 

\noindent
{\bf Problem IB}:  
Determine the image of 
$$
\bar{\beta}: \Bir^G(X) \to \mathrm{Out}(G). 
$$

 \ 

Using this terminology, we can compute the number of $G$-actions with a fixed image in $\Aut(X)$ as
$$
\frac{|\mathrm{Out}(G)|}{|\mathrm{image}(\bar{\alpha})|}. 
$$
Similarly, the number of actions with fixed image in $\Aut(X)$, up to conjugation in $\Bir(X)$, is given by
$$
\frac{|\mathrm{Out}(G)|}{|\mathrm{image}(\bar{\beta})|}. 
$$

Problem IA is purely group-theoretical, and its solution can be automated with computer algebra packages such as {\tt Magma}. On the other hand, Problem IB is of birational nature. If we knew the generators of $\Bir^G(X)$ then we could compute the image of $\bar{\beta}$. In practice, this is rarely feasible. 
However, in some situations, the solution is easy, for example,  when 
\begin{itemize}
    \item $\mathrm{Out}(G)$ is trivial, e.g., $G=\fS_n$,  with $n\neq 6$,  
    \item $\Aut(X)=\Bir(X)$, so that Problems IA and IB coincide, e.g., if $X$ is a curve or if the canonical class $K_X$ is ample, 
    \item $\Aut^G(X)=\Bir^G(X)$, 
    \item one can identify sufficiently many elements in $\Bir^G(X)$ to prove the surjectivity of $\bar{\beta}$. 
\end{itemize}

\begin{exam}
There is a unique $\fA_5\subset \PGL_3\simeq\Aut(\bP^2)$,
up to conjugation. We have $\mathrm{Out}(\fA_5) \simeq C_2$ and 
$$
\Aut^{\fA_5}(\bP^2)\simeq \fA_5.  
$$    
As there are no nontrivial homomorphisms $\fA_5\to C_2$, we conclude that 
there are two $\fA_5$-actions on $\bP^2$. However, 
$$
\Bir^{\fA_5}(\bP^2)\simeq \fS_5,  
$$
and the homomorphism 
$\bar{\beta}$ is surjective \cite{two-local}. 
It follows that both $\fA_5$-actions on $\bP^2$ are conjugated in $\Cr_2$; a realization of this conjugation is in 
\cite[Lemma 6.3.3]{CS}, see also Example~\ref{exam:a5}. 
\end{exam}
Additional examples with surjective $\bar{\beta}$ 
will be given in Section~\ref{sect:plane}. 

\subsection*{Equivariant Sarkisov program for rational surfaces}

Let $X$ be a smooth rational $G$-surface. The equivariant Minimal Model Program implies that $X$ is $G$-birational to a $G$-surface 
$S$ such that 
\begin{itemize}
\item $S$ is a smooth del Pezzo surface with $\rk \Pic(S)^G=1$, or
\item $S$ admits a $G$-equivariant conic bundle 
$$
\pi :S\to \bP^1
$$ 
with $\rk \Pic(S)^G=2$.     
\end{itemize}
Such $G$-birational models are called $G$-{\em Mori fiber spaces} (in dimension 2). 
Any $G$-birational map between such $G$-Mori fiber spaces can be factorized into a sequence of elementary $G$-birational maps, called {\em $G$-Sarkisov links}. 
We recall their classification, following the classification of Sarkisov links between 2-dimensional Mori fiber spaces over nonclosed ground fields in \cite{Isk}.  
There are three basic types:
\begin{itemize}
\item {\bf DP-DP} (del Pezzo to del Pezzo):

\centerline{
\xymatrix{
S & \ar[l]_{\sigma} \tilde{S} \ar[r]^{\sigma'} & S', 
}
}

\ 

\noindent
where 
\begin{itemize} 
\item 
$\sigma$ is a blowup of a $G$-orbit $\Sigma$ such that $|\Sigma|<K_S^2$ and $\tilde{S}$ is a del Pezzo surface, 
\item $S'$ is a (smooth) del Pezzo $G$-surface with $\rk\Pic(S')^G=1$, 
\item 
$\sigma'$ is a blowup of a $G$-orbit $\Sigma'$ such that $|\Sigma'|<K_{S'}^2$; 
\end{itemize}
\item {\bf DP-CB} (del Pezzo to conic bundle / and its inverse): 

\centerline{
\xymatrix{
S & \ar[l]_{\sigma} \tilde{S} \ar[r]^{\tilde{\pi}} & \bP^1 
}
}

\ 

\noindent
where 
\begin{itemize} 
\item $\sigma$ is a blowup of a $G$-orbit $\Sigma$ such that $|\Sigma|<K_S^2$ and $\tilde{S}$ is a del Pezzo surface, 
\item  $\tilde{\pi}$ is a $G$-equivariant conic bundle; 
\end{itemize}
    \item {\bf CB-CB} (conic bundle to conic bundle):
    
\centerline{
\xymatrix{
S  \ar[d]_{\pi}  & \ar[l]_{\sigma}\tilde{S} \ar[r]^{\sigma'} & S'\ar[d]^{\pi'} \\
\bP^1 \ar@{=}[rr] & & \bP^1
}
}

\noindent 
where 
\begin{itemize} 
\item  $\pi'$ is a $G$-equivariant conic bundle, with $K_{S'}^2=K_{S}^2$, 
\item $\sigma$ is a blowup of a $G$-orbit $\Sigma$ such that
no points of $\Sigma$ are contained in the singular fibers of $\pi$ and every smooth fiber of $\pi$ contains at most one point of $\Sigma$, 
\item $\sigma'$ is the blowdown of the strict transforms of the fibers of $\pi$ that contain points of $\Sigma$. 
\end{itemize}
\end{itemize}

\begin{exam}
\label{exam:sarkisov}
We list $G$-Sarkisov links that start from $S=\bP^2$:
\begin{itemize}
    \item {\bf DP-DP} with $r:=|\Sigma|\in \{ 2,3,5,6,7,8\}$ 
    \begin{itemize}
        \item $r=2$: $S'=\bP^1\times \bP^1$ and $\sigma'$ blows down the strict transform of the line passing through $\Sigma$,
        \item $r=3$: $S'=\bP^2$ and $\sigma'$ blows down the strict transforms of lines passing through pairs of points in $\Sigma$, 
        \item $r=5$: $S'$ is a del Pezzo surface of degree 5 and $\sigma'$ 
        blows down the strict transform of the conic through $\Sigma$, 
        \item $r=6$: $S'=\bP^2$ and $\sigma'$ blows down strict transforms of conics passing through 5 points of $\Sigma$, 
        \item $r=7$: $\tilde{S}$ is a del Pezzo surface of degree 2 and $\sigma'=\sigma\circ \tau$, where $\tau$ is the involution of the anticanonical double cover $\tilde{S}\to \bP^2$,   
        \item $r=8$: $\tilde{S}$ is a del Pezzo surface of degree 1 and $\sigma'=\sigma\circ \tau$, where $\tau$ is the involution of the double cover $\tilde{S}\to \bP(1,1,2)$ given by the linear system $|-2K_{\tilde{S}}|$.            
    \end{itemize}
    \item {\bf DP-CB} with $r=|\Sigma|\in \{ 1, 4\}$
    \begin{itemize}
        \item $r=1$: $\tilde{S}$ is the Hirzebruch surface $\mathbb F_1$  and $\tilde{\pi}$ a $\bP^1$-bundle, 
        \item $r=4$:  $\tilde{S}$ is a del Pezzo surface of degree 5 and  $\tilde{\pi}$ a conic bundle with 3 singular fibers. 
    \end{itemize}
\end{itemize}
\end{exam}

\begin{rema}
\label{rema:Bertini-Geiser}
Let $S$ be a smooth del Pezzo $G$-surface of degree $\geqslant 2$ such that $\rk\Pic(S)^G=1$. 
Suppose that there exists a $G$-Sarkisov link 
$$
\xymatrix{
S & \ar[l]_{\sigma} \tilde{S} \ar[r]^{\sigma'} & S'}
$$
such that $\sigma$ is a blowup of a $G$-orbit $\Sigma$ with
$$
|\Sigma|\in\{K_S^2-1,K_S^2-2\},
$$
where $S'$ is a (smooth) del Pezzo $G$-surface with $\rk\Pic(S)^G=1$, 
and $\sigma'$ is a blowup of a $G$-orbit $\Sigma'$. 
Then $\tilde{S}$ is a smooth del Pezzo surface of degree $1$ or $2$, and $\mathrm{Aut}(\tilde{S})$ contains an involution $\tau$ that centralizes the action of the group $G$. 
The involution $\tau$ is known as {\em Bertini} involution (when $K_{\tilde{S}}^2=1$) or {\em Geiser} involution (when $K_{\tilde{S}}^2=2$).
This implies that $S$ and $S'$ are $G$-birational, in particular, we have $|\Sigma'|=|\Sigma|$. 
Hence, we may assume that $\sigma'\circ\beta=\gamma$ for
$$
\gamma=\sigma\circ\tau\circ\sigma^{-1}.
$$
Note that $\gamma\in\mathrm{Bir}^G(S)$, and its image $\beta(\gamma)\in\mathrm{Aut}(G)$ is trivial, because $\gamma$ centralizes $G$.
In the following, we will say that $\gamma$ is a Bertini or Geiser (birational) involution of the surface $S$, respectively.
\end{rema}

A more detailed description of two-dimensional $G$-Sarkisov links is in \cite[Section 7]{DI}. 
For instance, if $S\to \bP^1$ is a $G$-conic bundle with $\rk\Pic(S)^G=2$,  
then $7\ne (-K_S)^2\le 8$ and the following assertions hold:
\begin{itemize}
    \item if $(-K_S)^2\le 1$, then $S$ is not $G$-birational to a del Pezzo surface, see \cite[Section 8]{DI},
    \item if $(-K_S)^2=3,5,6$, then $S$ is a del Pezzo surface, 
    \item if $(-K_S)^2=8$, then $S$ is $G$-birational to a del Pezzo surface.
\end{itemize}
If $(-K_S)^2=4$, there are examples such that $S$ is a del Pezzo surface, and there are examples such that $S$ is not $G$-birational to a smooth del Pezzo surface, see \cite[Section 8]{DI}. It was an open question, posed in \cite[Section 9]{DI} whether or not $S$ is $G$-birational to a smooth del Pezzo surface when $(-K_2)^2=2$. 
Even though this is not directly related to the main topic of this paper, we include an example that answers this question: 

\begin{exam}
\label{exam:dp2-conic}
Let $\bar{S}\subset \bP(1_x,1_y,3_z,3_w)$ be given by 
$$
zw=x^6+y^6, 
$$
let $\eta\colon S\to\bar{S}$ be the minimal resolution of singularities, 
and let $G$ be the subgroup in $\Aut(\bar{S})$ generated by 
$$
\sigma_{z,w}: (z,w)\mapsto (w,z), \quad \sigma_{x,y}: (x,y)\mapsto (x,y),
$$
$$
\theta_{12}: (x,y,z,w)\mapsto (\zeta_{12}x,\zeta_{12}^{-1}y,-z,w).
$$
Then $(-K_S)^2=2$, $\rk\Pic(S)^G=2$, $G\simeq (C_2)^2.\mathfrak{D}_6$, and we have the following $G$-equivariant commutative diagram:
$$
\xymatrix{
&S\ar@{->}[dr]^{\pi}\ar@{->}[dl]_{\eta}\\%
\bar{S}\ar@{-->}[rr]^{\xi} && \mathbb{P}^1_{x,y}}
$$
where $\pi$ is a conic bundle, and $\xi$ is given by $(x,y,z,t)\mapsto(x,y)$. 
By the classification of $G$-Sarkisov links, $S$ is not $G$-birational to a smooth del Pezzo surface unless there exists the $G$-equivariant commutative diagram:
$$
\xymatrix{
S\ar@{->}[d]_{\pi}\ar@{-->}[rr]&&S^\prime\ar@{->}[d]^{\pi^\prime}\\%
\mathbb{P}^1\ar@{=}[rr]&& \mathbb{P}^1}
$$
where $S^\prime$ is a del Pezzo surface of degree $2$ with $\rk\Pic(S^\prime)^G=2$, and $\pi^\prime$ is a $G$-conic bundle.
Using the classification of actions on del Pezzo surfaces of degree $2$ in \cite{DI}, we exclude the latter possibility, so that $S$ is not $G$-birational to a smooth del Pezzo surface.
\end{exam}

Now, we suppose that $S$ is a del Pezzo surface with $\rk\Pic(S)^G=1$. Such $S$ is called:
\begin{itemize}
    \item $G$-birationally super-rigid if there are no $G$-Sarkisov links starting from $S$,
    \item $G$-birationally rigid if every $G$-Sarkisov link that starts from $S$ ends at a del Pezzo surface which is $G$-biregular to $S$, 
    \item $G$-solid if $S$ is not $G$-birational to a $G$-conic bundle.
\end{itemize}

If $S$ is $G$-birationally super-rigid then 
$\Aut^G(S)=\Bir^G(S)$, and the number of $G$-actions on $S$ with fixed image in $\Aut(S)$ is the same, whether we consider it up to conjugation in $\Aut(S)$ or in $\Bir(S)$.

\begin{exam}
\label{exam:a6}
Let $S=\bP^2$ and $G=\fA_6$. Then $S$ does not have $G$-orbits $\Sigma$ of length $|\Sigma|< K_S^2=9$. By Example~\ref{exam:sarkisov},  $S$ is $G$-birationally super-rigid.      
\end{exam}

\begin{exam}
\label{exam:a5}
Let $S=\bP^2$ and $G=\fA_5$. Then $S$ contains a {\em unique} $G$-orbit $\Sigma$ of length $|\Sigma|< K_S^2=9$, this orbit has length $6$ and the points of $\Sigma$ are in general linear position and are not contained in a conic. By Example~\ref{exam:sarkisov}, $S'=\bP^2$, and is $G$-biregular to $S$, thus $S$ is $G$-birationally rigid.      
\end{exam}

\subsection*{Burnside formalism}
\label{sect:burn}

This formalism, introduced in \cite{BnG} and applied in, e.g., \cite{HKT-small},  \cite{KT-linear}, \cite{TYZ-3}, allows to distinguish birational types of actions of finite groups in many new situations. It takes into account information from 
the {\em stabilizer stratification} on a {\em standard} birational model for the action. This is a model $X$ such that
\begin{itemize} 
\item there is a Zariski open $U\subset X$ on which the $G$-action is free, 
\item the complement $X\setminus U $ is a normal crossings divisor such that for each irreducible component $D_{\alpha}$ in $X\setminus U=\cup_{\alpha\in\cA}D_\alpha$, and all $g\in G$, we have that
$$
g(D_{\alpha})=D_{\alpha}\quad \text{ or } \quad 
g(D_{\alpha})\cap D_{\alpha} =\emptyset.
$$
\end{itemize}
On such a model, generic stabilizers of all subvarieties are abelian. 

We describe the simplest version of the general theory in the case of surfaces: given a standard model as above, a $G$-orbit of irreducible components of $X\setminus U$, and a choice of a representative $D_{\alpha}$ of this orbit, we have a symbol
$$
(H_\alpha, Z_{\alpha} \actsfromleft k(D_{\alpha}), (b_{\alpha})),
$$
where 
\begin{itemize}
    \item $H_\alpha$ is the (cyclic) generic stabilizer of $D_\alpha$,
    \item $Z_{\alpha}\subseteq \mathrm Z_G(H_{\alpha})/H_{\alpha}$, a subgroup of the centralizer modulo stabilizer, acting generically freely on $D_\alpha$, 
    and 
    \item $b_\alpha\in H_{\alpha}^\vee$ is the weight of $H_{\alpha}$ in the normal bundle of $D_{\alpha}$. 
\end{itemize}
 Such symbols are considered up to $G$-conjugation. 

A symbol is called {\em incompressible} if $D_{\alpha}$ is a curve of genus $\ge 1$, or if $Z_{\alpha}$ is not cyclic. 
The class 
$$
[X\actsfromright G]:=\sum_{\alpha\in \mathcal A/G} (H_\alpha, Z_{\alpha} \actsfromleft k(D_{\alpha}), (b_{\alpha}))
$$
in the free abelian group 
$$
\Burn_2^{\mathrm{inc}}(G),
$$
spanned by incompressible symbols (up to conjugation), is an equivariant birational invariant of the $G$-action. 

In our application to Theorem~\ref{thm:intransBurn}, we have $D_{\alpha}=\bP^1$, and $Z_{\alpha}$ is one of the 
$$
\fD_n, \quad \fA_4, \quad \fS_4, \quad \fA_5. 
$$
The following example from \cite[Section 10]{KT-linear} and \cite[Section 7]{TYZ-3} illustrates this.  

\begin{exam}
\label{exam:burn-I}
For $t\ge 2$ and odd $n\ge 5$, consider 
$G=C_t\times \mathfrak D_n$. Let $\chi$ be a primitive character of $C_t$ and $\psi$ a primitive character of $C_n\subset \mathfrak D_n$. Let $V_{\psi}$ be the 2-dimensional faithful representation of $\mathfrak D_n$ induced from $\psi$, and $V_{\chi,\psi}=\chi\otimes V_{\psi}$. Then $G$ acts generically freely on 
$$
    \bP^2=\bP(V)=\bP(\mathbf{1}\oplus V_{\chi,\psi}).
    $$
The class 
$$
[\bP(V)\actsfromright G]\in \Burn_2^{\mathrm{inc}}(G)
$$
equals 
$$
(C_{t}, \mathfrak D_{n}\actsfromleft k(\bP^1), (\chi))
+ (C_{t}, \mathfrak D_{n}\actsfromleft k(\bP^1), (-\chi)).
$$
In particular, if $\chi\neq \pm \chi'$ or $\psi\neq \pm \psi'$, the corresponding actions on $\bP(V)$ and $\bP(V')$ are not birational. 
\end{exam}

\section{Classification of finite subgroups in $\PGL_3$}
\label{sect:plane}

Finite subgroups $G\subset \PGL_3$ have been classified by Blichfeldt \cite{blichfeldt}. These arise via projectivizations $\bP(V)$ of
faithful 3-dimensional representations $V$ of central  extensions  $\tilde{G}$ of $G$.
We recall the relevant terminology: a $G$-action
on $\bP^2=\bP(V)$ is called:
\begin{itemize}
\item[{\bf (I)}] {\em intransitive}: if $G$ fixes a point in $\bP^2$;
\item[{\bf (T)}] {\em transitive but imprimitive}: if $G$ does not fix a point in $\bP^2$ but there exists a $G$-orbit of length $3$;
\item[{\bf (P)}] {\em primitive}: neither of the above.
\end{itemize}

\

The following groups give rise to primitive actions:
$$
\fA_5, \quad \fA_6,\quad  \mathsf{PSL}_2(\bF_7),\quad  \mathsf{ASL}_2(\bF_3), \quad \mathsf{PSU}_3(\bF_2), \quad C_3^2\rtimes C_4.
$$
Each of these is unique in $\PGL_3$, up to conjugation. Explicit generators are given in \cite{Yau-Yu}.
The group $\mathsf{ASL}_2(\bF_3)$ is known as the {\em Hessian group} ---
after a suitable coordinate change, it leaves invariant the Hesse pencil:
$$
x^3+y^3+z^3+\lambda xyz=0, \quad \lambda\in\mathbb{P}^1,
$$
in $\mathbb{P}^2_{x,y,z}$. The action on the curves of this pencil gives an exact sequence of groups
$$
1\to C_3^2\rtimes C_2\to \mathsf{ASL}_2(\bF_3)\to\mathfrak{A}_4\to 1,
$$
where $C_3^2\rtimes C_2\subset \PGL_3$ is the subgroup that leaves invariant a general curve in the pencil,
and $C_3^2\subset C_3^2\rtimes C_2$ acts on a general elliptic curve in the pencil via translations by $3$-torsions.
This subgroup is transitive but not primitive, and $\mathsf{ASL}_2(\bF_3)$ can also be defined as its normalizer in $\PGL_3$.
Up to conjugation, the three primitive subgroups $C_3^2\rtimes C_4$, $\mathsf{PSU}_3(\bF_2)$, $\mathsf{ASL}_2(\bF_3)$ are nested as follows:
$$
C_3^2\rtimes C_4\subset\mathsf{PSU}_3(\bF_2)\subset\mathsf{ASL}_2(\bF_3).
$$

We describe groups giving rise to actions of type {\bf (T)}. In this case, $G$ has an orbit of length 3, consisting of points in linear general position. Changing coordinates, we may assume that these points are
$$
[1:0,0], \quad [0:1:0], \quad  [0:0:1].
$$
The $G$-action on these points gives rise to the exact sequence
$$
1\to T\to G\stackrel{\nu}{\lra} \fS_3,
$$
which turns out to be split. The kernel of $\nu$ consists of diagonal automorphisms.
The image is either $C_3$ or $\fS_3$.

\begin{prop}[{\cite[Theorem 4.7]{DI}}]
    \label{prop:trans}
Let 
$G\subset \PGL_3\simeq \Aut(\bP^2)$ be a finite subgroup giving rise to an action of type {\bf (T)}. Then,
up to conjugation in $\PGL_3$, one  of the following holds:
\begin{itemize}
\item $G\simeq C_n^2\rtimes C_3$ is generated by 
$$
\mathrm{diag}(\zeta_{n},1,1),\quad  \sigma_{123}.
$$  
    \item  $G\simeq C_n^2\rtimes \fS_3$ is generated by
    $$
\mathrm{diag}(\zeta_n, 1, 1), \quad \sigma_{123},\quad  \sigma_{12},
    $$
     \item 
        $G\simeq (C_n\times C_{n/r})\rtimes C_3$ is generated by
 $$
\mathrm{diag}(\zeta_n^r, 1, 1), \quad \mathrm{diag}(\zeta_n^s, \zeta_n, 1), \quad \sigma_{123},
    $$
    where $r\ge 1$, $r\mid n$, and $s^2-s+1\equiv 0 \pmod r$,
\item  $G\simeq (C_n\times C_{n/3})\rtimes \fS_3$ is generated by
    $$
\mathrm{diag}(\zeta_n^3, 1, 1), \quad \mathrm{diag}(\zeta_n^2, \zeta_n, 1), \quad \sigma_{123},\quad  \sigma_{12},
    $$
    with $3\mid n$,   
    \end{itemize}
where
$$
\sigma_{123}: (x,y,z)\to (y,z,x), \quad  \sigma_{12}: (x,y,z)\to (y,x,z),
$$
on $\bP^2_{x,y,z}$. Each of these subgroups is unique up to conjugation in $\PGL_3$.
\end{prop}
\begin{proof}
The classification is achieved in \cite{DI}. To show the uniqueness of each subgroup in $\PGL_3$, it suffices to show that the two choices of $s$ give rise to the same subgroup in the third case $G\simeq (C_n\times C_{n/r})\rtimes C_3$.

    Let  $G_1, G_2$ be groups corresponding to different choices of roots 
$s_1, s_2$, for fixed $n, r$, in the third case. Since $s_1+s_2=1\pmod r$, we have 
$$
\sigma_{12}\cdot\mathrm{diag}(\zeta_n^{s_1}, \zeta_n, 1)\cdot\sigma_{12}=\mathrm{diag}(\zeta_n^{s_1}, 1,\zeta_n).
$$
Up to multiplying by powers of $\mathrm{diag}(\zeta_n^r,1,1)$, we have 
$$
\mathrm{diag}(\zeta_n^{s_2},\zeta_n,1)\in \sigma_{12}G_1\sigma_{12}.
$$
It follows that 
$$
\sigma_{12}\cdot G_1\cdot\sigma_{12}=G_2.
$$
\end{proof}

We turn to actions of type {\bf (I)},
following \cite[Section 4.2]{DI}.
The existence of a fixed point implies that we have an embedding $G\hookrightarrow \GL_2$,
inducing an action on $\bP^1$, via the projection $\GL_2\to \PGL_2$; we denote by $\bar{G}$ the image of $G$, and by $C_m$ its kernel, a cyclic subgroup of order $m$.
We write 
$$
\chi_r : (x_1,x_2,x_3)\mapsto (x_1,\zeta_rx_2,\zeta_rx_3),
$$
for a primitive root of unity $\zeta_r$ of order $r\geq 1$. We have:

\begin{prop}[{\cite{GL2C}}]
    \label{prop:intrans}
Let $G\subset \PGL_3\simeq \Aut(\bP^2)$ be a finite subgroup giving rise to an action of type {\bf (I)}. Then,
up to conjugation in $\PGL_3$, one  of the following holds:
\begin{itemize}[wide, labelwidth=!, labelindent=0pt]
  \item $\bar G\simeq C_n$. $G$ is generated by 
$$
\chi_r, (x_1,\zeta_n^rx_2,\zeta_n^{m}x_3), \quad r,m\in\bZ_{\geq 0}.
$$

    \item $\bar G\simeq \fD_n$:
    \begin{itemize}
    \item
    $n$ odd. $G$ is generated by
    $$
\chi_r,(x_1,\zeta_nx_2,\zeta_n^{-1}x_3),(x_1,\zeta_{2^m}x_3,\zeta_{2^m}x_2),\quad r,m\in\bZ_{\geq0}.
$$
\item $n$ even. $G$ is generated by
    $$
\chi_r,(x_1,\zeta_{2n}x_2,\zeta_{2n}^{-1}x_3),(x_1,\zeta_{4}x_3,\zeta_{4}x_2),\quad r\in\bZ_{\geq0}.
$$ 
\item
   $n$ even. $G$ is generated by
    $$
\chi_r,(x_1,\zeta_{2^{m+1}}\zeta_{2n}x_2,\zeta_{2^{m+1}}\zeta_{2n}^{-1}x_3), (x_1,\zeta_{4}x_3,\zeta_{4}x_2),\quad r\in\bZ_{\geq 0},m\in\bZ_{\geq 1}.
$$
\item
   $n$ even,
   $n\ne 2$. $G$ is generated by
$$
\chi_r,(x_1,\zeta_{2n}x_2,\zeta_{2n}^{-1}x_3), (x_1,\zeta_{2^{m+1}}\zeta_{4}x_3,\zeta_{2^{m+1}}\zeta_{4}x_2),\quad r\in\bZ_{\geq 0},m\in\bZ_{\geq 1}.
$$
\end{itemize}
\item $\bar G\simeq \fA_4$.
\begin{itemize}
\item
$G$ is generated by
    $$
\chi_r,(x_1,\zeta_3^2x_2,\zeta_3x_3-x_2),(x_1,x_3,-x_2),\quad r\in\bZ_{\geq0},
$$
\item or  $G$
 is generated by
    $$
\chi_r,(x_1,\zeta_{3^{m+1}}\zeta_3^2x_2,\zeta_{3^{m+1}}(\zeta_3x_3-x_2)),(x_1,x_3,-x_2),\quad r,m\in\bZ_{\geq0}.
$$
\end{itemize}
 \item $\bar G\simeq \fS_4$.
 $G$ is generated by
    $$
\chi_r, (x_1,\zeta_{2^{m+1}}\zeta_{4}(-2x_2+s_1x_3),\zeta_{2^{m+1}}\zeta_4(s_2x_2+2x_3)),
$$
$$
(x_1,-s_2x_2-x_3,s_1x_2+(s_1+1)x_3),
$$
where
$$
s_1=\zeta_8^3+\zeta_8-1, \quad s_2=\zeta_8^3+\zeta_8+1, \quad
r,m\in\bZ_{\geq0}.
$$

\item $\bar G\simeq \fA_5$. $G$ is generated by
    $$
\chi_r,(x_1,(\zeta_{5}^3+\zeta_5^4)x_2-(\zeta_5^4+1)x_3,\zeta_5^3x_2+(\zeta_5^2+\zeta_5)x_3),(x_1,\zeta_5^2x_3,-\zeta_5^3x_2),
$$
for $r\in\bZ_{\geq0}.$
\end{itemize}
Moreover, different subgroups in the list are not conjugated in $\PGL_3$.
\end{prop}

From this description, we obtain:

\begin{coro}
\label{coro:nice}
If a group $G$ gives rise to actions of different classes among {\bf (I)}, {\bf (T)}, and {\bf (P)}, then $G\simeq C_3^2$, the actions are of type {\bf (I)} and {\bf (T)}, and
there are two subgroups in $\PGL_3$ isomorphic to $G$, up to conjugation, generated by 
 \begin{align*}
           \langle\mathrm{diag}(1,\zeta_3,1),\mathrm{diag}(1,1,\zeta_3)\rangle\quad \text{of Type {\bf (I)}},\\ \langle\mathrm{diag}(1,\zeta_3,\zeta_3^2),(x_3,x_1,x_2)\rangle\quad \text{of Type {\bf (T)}}.
        \end{align*}

\end{coro}
\begin{proof}
    We start by describing the center $Z(G)$ if $G$ is in class 
    \begin{itemize}
        \item [{\bf (I)}:] if $G$ is nonabelian, the quotient $G/Z(G)$ is a noncyclic subgroup of $\PGL_2$, and $Z(G)\supset C_2$ unless $G=C_r\times\fD_n$, with $n,r$ odd.
        \item [{\bf (T)}:] the center $Z(G)=1, C_3$, or $C_3^2$. 
        \item[{\bf (P)}:] the center $Z(G)=1$, unless $G=C_3^2\rtimes C_4$, where $Z(G)=C_2$ and $G/Z(G)=C_3\rtimes \fS_3$.
    \end{itemize}
We conclude that no group giving rise to actions in type {\bf (P)} also gives rise to an action of type {\bf (I)} or {\bf (T)}. Assume that $G$ gives rise to actions of both types {\bf (I)} and {\bf (T)}. If $G$ is nonabelian, we know that $G=C_3\times \fD_n$ with $n$ odd. But no group giving rise to actions of type {\bf (T)} is isomorphic to this group. It follows that $G$ is abelian and $G=Z(G)$. So $|G|$ is a power of 3. It is not hard to see that the only possibility is $G=C_3^2$ and it has two conjugacy classes in $\PGL_3$, as is explained above.
\end{proof}

\begin{rema} 
Finite isomorphic subgroups of $\PGL_2\simeq\Aut(\bP^1)$ are conjugated, see, e.g.,  \cite[Proposition 4.1]{beau-sub}. Corollary~\ref{coro:nice} shows that this does not hold for $\PGL_3$. 
\end{rema}

We proceed to give a description of conjugacy classes of finite subgroups in $\PGL_3$.

\begin{coro}\label{coro:conjugatePGL3}
    Let $G_1$ and $G_2$ be isomorphic finite nonabelian subgroups of $\PGL_3$. Then $G_1$ and $G_2$ are conjugated in $\PGL_3$ unless
        \begin{align*}
            G_1=\langle\mathrm{diag}(\zeta_{n},1,1),(x_1,\zeta_3^2x_2,\zeta_3x_3-x_2),(x_1,x_3,-x_2)\rangle,\\
            G_2=\langle\mathrm{diag}(\zeta_{n},1,1),(x_1,x_2,\zeta_3^2x_3-\zeta_3x_2),(x_1,x_3,-x_2)\rangle,
        \end{align*}
    for $n\geq 1$. In this case, $G_1$ and $G_2$ are not conjugated in $\PGL_3$, and 
    $$
    G_1\simeq G_2\simeq \begin{cases}
        C_n\rtimes \fA_4&\text{when $n$ is even,} \\
        C_{2n}\rtimes\fA_4&\text{when $n$ is odd}.\\
    \end{cases}
    $$
\end{coro}
\begin{proof}
    By Corollary~\ref{coro:nice}, we know that nonabelian groups giving rise to actions of different types among {\bf (I)}, {\bf (T)}, and {\bf (P)} are not isomorphic. Each subgroup in type {\bf (T)} and {\bf (P)} is unique in $\PGL_3$, up to conjugation. It follows that $G_1$ and $G_2$ give rise to actions of type {\bf (I)}. Using Proposition~\ref{prop:intrans}, we 
    check that the only possibility is when $\bar G=G/Z(G)\simeq \fA_4$, with the generators given above. 
    
    To show $G_1$ and $G_2$ are not conjugated in $\PGL_3$, we note that when $n=1$, the generic stabilizer of the line $\{x_2=(2\zeta_3+ 1)x_3\} \subset\bP^2_{x_1,x_2,x_3}$ is trivial under the action induced by $G_1$, but is $C_3$ under that induced by $G_2$.
\end{proof}

\section{Primitive actions}
\label{section:primitive}
Equivariant birational geometry of {\em primitive} actions was essentially settled, via equivariant MMP, in \cite{sako}; the classification of conjugacy of actions was not addressed. We proceed to fill this gap:

\begin{itemize}
    \item super-rigid, in particular, the actions are not conjugated in $\Cr_2$:
    \begin{itemize}
    \item $\fA_6$: four actions on $\bP^2$ \cite[Theorem B.9]{two-local},
    \item $\PSL_2(\bF_7)$: two actions on $\bP^2$ \cite[Theorem B.8]{two-local},
    \item  $\mathsf{PSU}_2(\bF_3)$: two actions on $\bP^2$ \cite[Proposition 3.16]{sako},
    \item  $\mathsf{ASL}_2(\bF_3)$: two actions on $\bP^2$ \cite[Proposition 3.16]{sako},
    \end{itemize}
    \item rigid:
    \begin{itemize}
\item
  $\fA_5$: two actions on $\bP^2$, conjugated in $\Cr_2$ and $\mathrm{Bir}^G(\mathbb{P}^2)\simeq\mathfrak{S}_5$ \cite[Lemma B.13]{two-local},
  \item  a subgroup $G=C_3^2\rtimes C_4$ of the Hessian group: two actions on $\bP^2$, conjugated in $\Cr_2$.
    \end{itemize}
\end{itemize}

We explain the last entry: there are eight faithful 3-dimensional representations of the central extension $C_3.G$, but only two, up to multiplication with characters, thus there are two $G$-actions on $\bP^2$. 
The outer automorphism group $\Out(G)$ is $C_2^2.$ In the automorphism group $\Aut(G)$, there are three conjugacy classes of subgroups of order 72, isomorphic to
$$
\mathfrak F_9,
\quad \fS_3\wr\fS_2,\quad\mathsf{PSU_3(\bF_2)}.
$$
Recall that we have inclusions $G\subset\mathsf{PSU_3(\bF_2)}\subset\PGL_3$. In particular, 
$$
\mathrm{Aut}^G(\bP^2)=\mathsf{PSU_3(\bF_2)}.
$$
Then the conjugation action of $\mathsf{PSU_3(\bF_2)}$ on $G$ induces a map
$$
\mathsf{PSU_3(\bF_2)}\stackrel{f}{\to} C_2\subset\Out(G).
$$

The conjugation in $\Cr_2$ can be seen as follows: according to \cite[Proposition 3.16]{sako}, there are exactly two orbits of size 6 on $\bP^2$, for each of the actions.
Blowing up any of these orbits, we obtain a cubic surface $S$,
which must be the Fermat cubic surface \cite[Table 4]{DI}, yielding an embedding
$$
G\hookrightarrow \Aut(S)=C_3^3\rtimes \fS_4.
$$
There is a unique subgroup $G_1\subset \Aut(S)$ of order 72. We know that $G_1\simeq\fS_3\wr\fS_2$ and $G\subset G_1$.
Choosing an element in $G_1\setminus G$, we obtain at element $\tau_1\in\Bir^G(\bP^2)$ such that $\langle\tau_1,G\rangle\simeq G_1$.
Similarly, blowing up another orbit of length $6$, we obtain at element $\tau_2\in\Bir^G(\bP^2)$ such that $\langle\tau_2,G\rangle\simeq G_1$.
Then
$$
\Bir^G(\bP^2)=\langle\tau_1,\tau_2,\mathsf{PSU}_3(\bF_2)\rangle,
$$
and $\langle\tau_1,\tau_2,G\rangle$ is an amalgamated product of two copies of $G_1$ with common subgroup $G$.
In particular, $\Bir^G(\bP^2)$ is infinite.

The conjugation action of $G_1$ on $G$ induces a map
$$
\fS_3\wr\fS_2\stackrel{g}{\to} C_2\subset\Out(G),
$$
and the images of $f$ and $g$ generate $\Out(G),$ since $\fS_3\wr\fS_2$ and $\mathsf{PSU_3(\bF_2)}$ are distinct index-2 subgroups of $\Aut(G)$. It follows that the two linear actions of $G$ are conjugated in $\Cr_2$.

\section{Transitive non-primitive actions}
\label{sect:trans}

In this section, we assume that $G\subset \PGL_3$ is one of the subgroups listed in Proposition~\ref{prop:trans}.
Observe that the points
\begin{equation}
\label{eq:orbit-3}    
[1:0:0], \quad [0:1:0], \quad [0:0:1]
\end{equation}
form a $G$-orbit of length $3$, and that 
$$
\iota\in \mathrm{Bir}^G(\mathbb{P}^2),
$$
where $\iota$ is the standard Cremona involution \eqref{eqn:crem}. 

If $G\simeq\fA_4$,  then $\mathrm{Aut}^G(\mathbb{P}^2)\simeq\fS_4$ and 
$$
\bar{\beta}(\mathrm{Aut}^G(\mathbb{P}^2))=\mathrm{Out}(G)\simeq C_2.
$$
Similarly, if $G\simeq\fS_4$, then $\mathrm{Aut}^G(\mathbb{P}^2)=G$ and $\mathrm{Out}(G)$  is trivial.
Hence, in these two cases, there is a unique $G$-action on $\mathbb{P}^2$.

\begin{lemm}[Pinardin]
\label{lemm:Antoine}
Suppose that $G\simeq\fS_4$. Then 
$$
\mathrm{Bir}^G(\mathbb{P}^2)=\langle G,\iota\rangle\simeq\fS_4\times C_2.
$$
\end{lemm}

\begin{proof}
Observe that the points 
\begin{equation}
\label{eq:orbit-4}    
[1:1:1],\quad [-1:1:1], \quad [1:-1:1], \quad [1:1:-1],
\end{equation}
form a $G$-orbit of length $4$. 
Then \eqref{eq:orbit-3} and \eqref{eq:orbit-4} are the only $G$-orbits in $\mathbb{P}^2$ of length $<9$ that are in general position.
Thus, every $G$-Sarkisov link that starts at $\mathbb{P}^2$ must blow up one of these two orbits.

Let $\sigma\colon\tilde{S}\to\mathbb{P}^2$ be the blow up of the $G$-orbit \eqref{eq:orbit-4}.
Then $\tilde{S}$ is a del Pezzo surface of degree $4$, and we have the following $G$-Sarkisov link:
$$
\xymatrix{
&\ar[ld]_{\sigma} \tilde{S} \ar[rd]^{\tilde{\pi}} \\
\bP^2 & & \bP^1}
$$
where $\tilde{\pi}$ is a conic bundle. 
The kernel of the $G$-action on $\bP^1$ is isomorphic to $C_2^2$, which implies that there are no $G$-Sarkisov links that start at $\tilde{S}$ except for the inverse of the constructed $G$-Sarkisov link, which brings us back to $\mathbb{P}^2$.
The required assertion follows from the classification of equivariant Sarkisov links starting at $\bP^2$.
\end{proof}

\begin{lemm}
\label{lemm:A4}
Suppose that $G\simeq\fA_4$. For $\lambda\in k\setminus\{0,1\}$ with $\lambda^6\ne -1$, let $\vartheta_\lambda$ be the birational selfmap of $\mathbb{P}^2$ given by
$$
\vartheta_\lambda: (x_1,x_2,x_3)\mapsto (f_1,f_2,f_3),
$$
where 
\begin{align*}
f_1&=x_1^4x_3 -\frac{\lambda^{12}+ 1}{\lambda^2}x_1^2x_2^2x_3 + \lambda^8x_2^4x_3 - 2\lambda^2x_1^2x_3^3 - 
    2\lambda^6x_2^2x_3^3 + \lambda^4x_3^5,\\
f_2&=\lambda^8x_1^4x_2 - 2\lambda^6x_1^2x_2^3 + \lambda^4x_2^5 -\frac{\lambda^{12}+ 1}{\lambda^2}x_1^2x_2x_3^2 - 
    2\lambda^2x_2^3x_3^2 + x_2x_3^4,\\
f_3&=\lambda^4x_1^5 - 2\lambda^2x_1^3x_2^2 + x_1x_2^4 - 2\lambda^6x_1^3x_3^2 -\frac{\lambda^{12}+ 1}{\lambda^2}x_1x_2^2x_3^2 + \lambda^8x_1x_3^4.
\end{align*}
Then $\vartheta_\lambda$ is an involution in $\mathrm{Bir}^G(\mathbb{P}^2)$, so, in particular, $\mathrm{Bir}^G(\mathbb{P}^2)$ is infinite.
Moreover, birational involutions $\vartheta_\lambda$, the standard Cremona involution $\iota$, and
$\mathrm{Aut}^G(\mathbb{P}^2)\simeq\fS_4$ generate the group $\mathrm{Bir}^G(\mathbb{P}^2)$.
\end{lemm}

\begin{proof}
Observe that $\mathbb{P}^2$ contains one $G$-orbit of length $3$, 
three $G$-orbits of length $4$, and infinitely many $G$-orbits of length $6$ such that each of them is a $G$-orbit of $[0:1:\lambda]$, for $\lambda\in k^\times$.
Let $\sigma_\lambda\colon S_\lambda\to\mathbb{P}^2$ be the blow up of the $G$-orbit of the point $[0:1:\lambda]$.
Then $S_\lambda$ is a smooth cubic surface if and only if $\lambda\ne 0,1$ and $\lambda^6\ne -1$.
In this case, we have the following $G$-commutative diagram:
$$
\xymatrix{
S_\lambda\ar[d]_{\sigma_\lambda}\ar@{->}[rr]^{\theta_\lambda}&&\ar[d]^{\sigma_\lambda}S_\lambda \\
\bP^2\ar@{-->}[rr]_{\vartheta_\lambda} & & \bP^2}
$$
where $\theta_\lambda$ is a biregular involution in $\mathrm{Aut}(S_\lambda)$.
Note also that 
$$
\langle G,\vartheta_\lambda\rangle\simeq\fS_4.
$$
Now, arguing as in the proof of Lemma~\ref{lemm:Antoine}, we see that $\mathrm{Bir}^G(\mathbb{P}^2)$ is generated by $\mathrm{Aut}^G(\mathbb{P}^2)$, the involution $\iota$, and involutions $\vartheta_\lambda$.
\end{proof}

Recall from \cite[Theorem 1.3]{sako} that  $\bP^2$ is $G$-birationally rigid if and only $G\not\simeq \fA_4,\fS_4$.
Moreover, if 
$$
G\not\simeq \fA_4, \quad \fS_4, \quad C_3^2, \quad C_3\rtimes\mathfrak{S}_3, \quad C_7\rtimes C_3,
$$
then it follows from  \cite[Lemma~3.14]{sako} that \eqref{eq:orbit-3} is the only $G$-orbit in $\mathbb{P}^2$ of length $3$. From the proof of \cite[Theorem 1.3]{sako}, we know that  $\mathrm{Bir}^G(\mathbb{P}^2)$ is finite, and
that 
\begin{equation}
\label{eq:Bir-Aut}
\mathrm{Bir}^G(\mathbb{P}^2)=\langle\mathrm{Aut}^G(\mathbb{P}^2),\iota\rangle.
\end{equation}

If $G\simeq C_3\rtimes \fS_3$ or $G\simeq C_3^2$, then it follows from \cite[Lemma~3.14]{sako} that $\mathbb{P}^2$ has $4$ orbits of length $3$ (the group $C_3\rtimes\mathfrak{S}_3$ is erroneously omitted in \cite[Lemma~3.14]{sako}).
In these two cases, we have 
$$
\mathrm{Aut}^G(\mathbb{P}^2)\simeq\mathsf{ASL}_2(\bF_3),
$$
and all $G$-orbits of length $3$ form one orbit of $\mathrm{Aut}^G(\mathbb{P}^2)$. 
It follows from \cite[Theorem 1.3]{sako} that \eqref{eq:Bir-Aut} also holds.

\begin{lemm}
\label{lemm:C3-S3}
Suppose $G\simeq C_3^2$ or $C_3\rtimes\mathfrak{S}_3$. Then $\mathrm{Bir}^G(\mathbb{P}^2)$ is infinite.
\end{lemm}

\begin{proof}
Set $\hat{G}=\mathrm{Bir}^G(\mathbb{P}^2)$.
If $\hat{G}$ is finite, its birational action on $\mathbb{P}^2$ is regularized on a smooth projective rational surface $S$ such that 
\begin{itemize}
    \item $S$ is a smooth del Pezzo surface with $\rk \Pic(S)^{\hat{G}}=1$, or
    \item $S$ admits a $\hat{G}$-equivariant conic bundle 
    $$
\pi :S\to \bP^1,
$$ 
with $\rk\, \Pic(S)^{\hat{G}}=2$.     
\end{itemize}
The latter case is impossible, since $\mathbb{P}^2$ is $G$-solid \cite{sako},
so that $S$ is a smooth del Pezzo surface.
Using the classification of finite subgroups of automorphism groups of smooth del Pezzo surfaces \cite{DI}, we conclude that $S\simeq\mathbb{P}^2$.
Since $\mathrm{Aut}^G(\mathbb{P}^2)\simeq\mathsf{ASL}_2(\bF_3)$,
it follows from Corollary~\ref{coro:nice} that $\hat{G}$ is a primitive subgroup in $\mathrm{Aut}(S)\simeq\PGL_3$, 
which is a contradiction since $\mathsf{ASL}_2(\bF_3)$ is a maximal subgroup in $\PGL_3$.
\end{proof}

If $G\simeq C_7\rtimes C_3$, then $\mathrm{Aut}^G(\mathbb{P}^2)=G$, and \eqref{eq:orbit-3} is the only $G$-orbit in $\mathbb{P}^2$ of length $3$.
In this case, it follows from the proof of \cite[Theorem 1.3]{sako} that 
$$
\mathrm{Bir}^G(\mathbb{P}^2)=\langle G,\iota,\tau_1,\tau_2,\tau_3\rangle,
$$
where $\tau_1,\tau_2,\tau_3$ are Geiser involutions, implicitly described in the proof of \cite[Lemma~3.14]{sako}.
One can show that the group $\langle\iota,\tau_1,\tau_2,\tau_3\rangle$ is infinite, which gives:

\begin{coro}
If $G\simeq C_7\rtimes C_3$, then $\mathrm{Bir}^G(\mathbb{P}^2)$ is infinite.    
\end{coro}

If $G\simeq\fA_4$ or $G\simeq\fS_4$, then $\bar{\beta}(\iota)$ is the identity, so that conjugation by $\iota$ does not change the $G$-action,
which is not very surprising, since $\fA_4$ and $\fS_4$ admit one faithful action on $\mathbb{P}^2$.
In all other cases, it follows from Remark~\ref{rema:Bertini-Geiser} that
$$
\bar{\beta}(\mathrm{Bir}^G(\mathbb{P}^2))=\langle\bar{\beta}(\mathrm{Aut}^G(\mathbb{P}^2)),\bar{\beta}(\iota)\rangle,
$$
where $\iota$ is the standard Cremona involution. 
We proceed to show how to compute $\bar{\beta}(\mathrm{Bir}^G(\mathbb{P}^2))$ in several representative examples.

\begin{exam}
\label{exam:C3-C3}
Suppose that $G\simeq C_3^2$. Then $\mathrm{Out}(G)\simeq\mathsf{GL}_2(\bF_3)$, the group $\mathrm{Aut}^G(\mathbb{P}^2)$ is the Hessian subgroup $\mathsf{ASL}_2(\bF_3)\subset\PGL_3$ and
$$
\bar{\beta}(\mathrm{Aut}^G(\mathbb{P}^2))\simeq \mathsf{SL}_2(\bF_3),
$$
and we have two non-isomorphic $G$-actions on $\mathbb{P}^2$. 
These actions are conjugated in $\Cr_2$, because $\bar{\beta}(\iota)\not\in\bar{\beta}(\mathrm{Aut}^G(\mathbb{P}^2))$.
Thus, 
$$
\bar{\beta}(\mathrm{Bir}^G(\mathbb{P}^2))=\mathrm{Out}(G),
$$
and there is a unique imprimitive $G$-action up to conjugation in $\Cr_2$.
\end{exam}

\begin{exam}
Suppose that $G\simeq  C_3\rtimes \fS_3$. Then $\mathrm{Aut}^G(\mathbb{P}^2)\simeq\mathsf{ASL}_2(\bF_3)$,
$\mathrm{Out}(G)\simeq\fS_4$,  and
$$
\bar{\beta}(\mathrm{Aut}^G(\mathbb{P}^2))\simeq \fA_4.
$$
Therefore, there are two non-isomorphic $G$-actions on $\mathbb{P}^2$. 
They are conjugated in $\Cr_2$, because $\bar{\beta}(\iota)\not\in\bar{\beta}(\mathrm{Aut}^G(\mathbb{P}^2))$. Then $$
\bar{\beta}(\mathrm{Bir}^G(\mathbb{P}^2))=\mathrm{Out}(G).
$$
\end{exam}

\begin{exam}
Suppose that $G\simeq C_7\rtimes C_3$. Then $\mathrm{Out}(G)\simeq C_2$ and
$$
\mathrm{Aut}^G(\mathbb{P}^2)=G.
$$
Hence, there are two non-isomorphic $G$-actions on $\mathbb{P}^2$, conjugated in $\Cr_2$ by $\iota$, and 
$$
\bar{\beta}(\mathrm{Bir}^G(\mathbb{P}^2))=\mathrm{Out}(G),
$$
since $\beta(\iota)$ is not an inner automorphism of $G$.
\end{exam}

\begin{exam}
In the notation of Proposition~\ref{prop:trans},
suppose further that $\nu(G)=\fS_3$ and $G\simeq C_n^2\rtimes \fS_3$ with $n\geq 3$. Then 
$$
\mathrm{Aut}^G(\mathbb{P}^2)=G,
$$
and $\mathrm{Out}(G)\simeq (\mathbb{Z}/n\mathbb{Z})^\times$, and we have $|(\mathbb{Z}/n\mathbb{Z})^\times|$ 
non-isomorphic $G$-actions on $\mathbb{P}^2$.
Then 
$$
\bar{\beta}(\mathrm{Bir}^G(\mathbb{P}^2))\simeq C_2,
$$
since $\beta(\iota)$ is not an inner automorphism.
\end{exam}

In all these examples, the conjugation by $\iota$ changes the action of the~group $G$ on $\mathbb{P}^2$. This is not a coincidence: 

\begin{lemm}
\label{lemm:iota-beta}
Suppose that $G\not\simeq\fA_4, \fS_4$.
Then $\bar{\beta}(\iota)\not\in\bar{\beta}(\mathrm{Aut}^G(\mathbb{P}^2))$.
\end{lemm}

\begin{proof}
Recall that $G$ is one of the subgroups in Proposition~\ref{prop:trans}.
Let us use notation introduced in this proposition. 
Keeping in mind examples above, we may assume further that $G\not\simeq C_3^2, C_3\rtimes \fS_3$. 
Then  the~$G$-orbit \eqref{eq:orbit-3} is the only $G$-orbit of length $3$ in $\mathbb{P}^2$,
so that \eqref{eq:orbit-3} is also an orbit of the group $\mathrm{Aut}^G(\mathbb{P}^2)$.

Denote by $T\subseteq G$ the subgroup consisting of transformations given by diagonal matrices. 
Then $G$ is generated by $T$, $\sigma_{123}$ and in some cases by $\sigma_{12}$, where all possibilities for $T$ are given by Proposition~\ref{prop:trans}.
Conjugation by $\iota$ gives an automorphism $\phi\in\mathrm{Aut}(G)$ that acts trivially on $\sigma_{123}$ and $\sigma_{12}$ (if $\sigma_{12}\in G$),
but 
$$
\phi(t)=\frac{1}{t},
$$
for every $t\in T$. Hence, $\phi$ is not trivial, since $G\not\simeq\fA_4, \fS_4$.

Let $G^\prime=\mathrm{Aut}^G(\mathbb{P}^2)$,
and let $T^\prime\subseteq G^\prime$ be the subgroup containing all transformations given by diagonal matrices.
Note that $\sigma_{123}\in G^\prime$, since $G\subset G^\prime$.
Similarly, if $\sigma_{12}\in G$, then $\sigma_{12}\in G^\prime$, so that
$$
G^\prime=\langle T^\prime, \sigma_{123}, \sigma_{12}\rangle.
$$
If $\sigma_{12}\not\in G$, then it follows from the proof of Proposition~\ref{prop:trans} that 
$$
G^\prime=\langle T^\prime, \sigma_{123}\rangle,
$$
unless $\nu(G)=C_3$ and $G\simeq (C_n\times C_{n/3})\rtimes C_3$ is generated by
$$
\mathrm{diag}(\zeta_n^3, 1, 1), \quad \mathrm{diag}(\zeta_n^2, \zeta_n, 1), \quad \sigma_{123},
$$
where $3\mid n$. In this case, $\sigma_{12}\in G^\prime$, and $G^\prime=\langle T^\prime, \sigma_{123}, \sigma_{12}\rangle$.

Now, we suppose that $\phi$ is given by a conjugation by some $g\in G^\prime$.
Then $g=t^\prime\sigma$ for some $\sigma\in\langle\sigma_{123},\sigma_{12}\rangle$. 
Note that $\sigma$ is not an identity, since conjugation by an element in $T^\prime$ acts trivially on $T$.
If $g=t^\prime\sigma_{123}$, then
$$
\frac{1}{t}=\phi(t)=t^\prime\sigma_{123}t(t^\prime\sigma_{123})^{-1}=\sigma_{123}t(\sigma_{123})^{-1},
$$
for every $t\in T$, which is a contradiction, since $\sigma_{123}t(\sigma_{123})^{-1}=t$ if and only if $t$ is an identity, 
and $T$ is nontrivial.
Similarly, if $g=t^\prime\sigma_{12}$, then
$$
\frac{1}{t}=\phi(t)=t^\prime\sigma_{12}t(t^\prime\sigma_{12})^{-1}=\sigma_{12}t(\sigma_{12})^{-1},
$$
for every $t\in T$. This implies that every element in $T$ is given by the diagonal matrix 
$$
\mathrm{diag}(a,a^{-1},1),
$$
for some $a\in k^\times$, which is only possible when $G\simeq C_3^2$. But this case is treated in Example~\ref{exam:C3-C3} and excluded, by assumption. 
Similarly, we obtain a contradiction for other possibilities for $\sigma\in\langle\sigma_{123},\sigma_{12}\rangle$.
\end{proof}

Summarizing, we obtain:

\begin{theo}
\label{thm:trans}
Let $\phi_1,\phi_2: G\hookrightarrow \PGL_3$ be different imprimitive actions such that $\phi_1(G)=\phi_2(G)$ is one of the subgroups in Proposition~\ref{prop:trans}. Then $\phi_1$ and $\phi_2$ are conjugated in $\Cr_2$ if and only if they are conjugated by the standard Cremona involution $\iota$.
Moreover, if $G\not\simeq\fA_4, \fS_4$, then conjugation by $\iota$ changes the isomorphism class of the $G$-action on $\mathbb{P}^2$.
\end{theo}

\section{Intransitive}
\label{sect:intran}

In this section, we prove Theorem~\ref{thm:intran}. We assume that $G\subset \PGL_3$ is one of the subgroups listed in Proposition~\ref{prop:intrans}.
Recall that the $G$-action on $\bP^2=\bP(V)$ arises from a faithful linear representation of $G$, which decomposes as $V=\mathbf 1\oplus V_2$. Let $C_t$ be the generic stabilizer of the $G$-action on the line $\mathfrak l=\bP(V_2)\subset \bP(V)$, and $\chi$ the character of $C_t$ in the normal bundle of $\mathfrak{l}$.

If $t=1$, for two such $G$-actions $\bP(V)$ and $\bP(V')$ arising from
$$
V=\mathbf 1\oplus V_2,\quad V'=\mathbf 1\oplus V_2',
$$
we have $G$-equivariant birationalities
$$
\bP(V)\sim\bP(\mathbf1\oplus\mathbf1)\times\bP(V_2)\sim\bP(V_2')\times\bP(V_2)\sim\bP(\mathbf1\oplus\mathbf1)\times\bP(V_2')\sim\bP(V'),
$$
 by the no-name lemma. In fact, if $G$ is nonabelian, we know that $G=\fD_n$, with odd $n$, when $t=1$. 
 
If $t\ge 2$, then the class of the action
\begin{align}\label{eqn:symbol}
     [\bP(V)\actsfromright G] \in \Burn_2^{\mathrm{inc}}(G)
\end{align}
equals
\begin{align}\label{eqn:inc}
(C_{t}, \bar G\actsfromleft k(\mathfrak l), (\chi)) +
(C_{t}, \bar G\actsfromleft k(\mathfrak l), (-\chi)),
\end{align}
where
$\bar G=G/C_t$
acts generically freely on $\mathfrak l=\bP(V_2)$. These symbols are
incompressible, provided $\bar G$ is nonabelian. The main result in this section is

\begin{theo}[cf. Theorem~\ref{thm:intran}]\label{thm:intransBurn}
    Intransitive $G$-actions with $t\geq 2$ and $G$ nonabelian are birational $\iff$ the corresponding $\bar G$-actions on the invariant line $\mathfrak l$ are isomorphic and the corresponding
characters $\chi$ are equal, up to $\pm1$.
\end{theo}
\begin{proof}[Proof of $(\Leftarrow)$]
Assume that the $\bar G$-action and $\chi$ do not satisfy the assumption. Then the incompressible symbols \eqref{eqn:inc} in the classes \eqref{eqn:symbol} of the two $G$-actions are different in $\Burn_2^{\mathrm{inc}}(G)$. The Burnside formalism implies that the two actions are not birational, see also \cite[Section 7]{TYZ-3}. 
\end{proof}

The rest of this section is devoted to a proof of the $(\Rightarrow)$ direction of Theorem~\ref{thm:intransBurn}. The proof  of Theorem~\ref{thm:intransBurn} implies Theorem~\ref{thm:intran}, namely, that the Burnside formalism is decisive for nonabelian intransitive actions. It suffices to identify actions with the same
$C_t, \pm \chi$, and the same $\bar G$-action on $\bP^1$. We do this via a case-by-case analysis of $\bar G$, using the list of all possibilities for $G$ in Proposition~\ref{prop:intrans}. 
\begin{rema}
   When $\bar G=\fA_4$ or $\fS_4$, there is a unique $\bar G$-action on $\bP^1$. Theorem~\ref{thm:intransBurn} then implies that $G$-actions are not birational if and only if the corresponding characters $\chi$ do not differ by $\pm1$.
\end{rema}

\begin{rema}
    For intransitive $G\subset \PGL_3$, the group $\Bir^G(\bP^2)$ is always infinite -- it contains a 1-dimensional torus generated by 
    $$
    \{\mathrm{diag}(1,a,a):a\in k^\times\}.
    $$
\end{rema}

\subsection{When $\bar G=C_n$}
In this case, $G$ is abelian. Linear $G$-actions are birational if and only if the corresponding invariants in \cite{reichsteinyoussininvariant} coincide. 
In fact, birational classification of linear actions of {\em abelian} groups has been settled, in all dimensions, in \cite[Theorem 7.1]{reichsteinyoussininvariant}.

\subsection{When $\bar G=\fD_n$ with odd $n$}
Recall that all $G$-actions are conjugated in $\Cr_2$ if the generic stabilizer $C_t$ of the $G$-action on $\bP(V_2)$ is trivial.
Hence, we assume that $t\geqslant 2$.
In each case, the standard Cremona involution $\iota$ is contained in $\mathrm{Bir}^G(\mathbb{P}^2)$.
It would be interesting to describe generators of $\mathrm{Bir}^G(\mathbb{P}^2)$.

\begin{proof}[Proof of Theorem~\ref{thm:intransBurn} in this case]
By Proposition~\ref{prop:intrans}, the action of $G$ depends on the choice of primitive weights $t_1,t_2,t_3$ in 
$$
\mathrm{diag}(1,\zeta_r^{t_1},\zeta_r^{t_1}),\quad \mathrm{diag}(1,\zeta_{2n}^{t_2},\zeta_{2n}^{-t_2}),\quad  (x_1,\zeta_{2^m}^{t_3}x_3,\zeta_{2^m}^{t_3}x_2).
$$
The action of the generic stabilizer of  $\mathfrak l=\{x=0\}\subset\bP^2$ is generated by 
$$
\mathrm{diag}(1,\zeta_r^{t_1},\zeta_r^{t_1}),\quad \mathrm{diag}(1,\zeta_{2^{m-1}}^{t_3},\zeta_{2^{m-1}}^{t_3}).
$$
The weights can be changed birationally as follows:
\begin{itemize}
    \item The standard Cremona involution $\iota$ composed with transposition of $x_2,x_3$ preserves $t_2$ and inverts 
$$
t_1\mapsto -t_1,\quad t_3\mapsto -t_3,
$$
\item the transposition of $x_2,x_3$ preserves $t_1,t_3$ and inverts $t_2\mapsto -t_2$,
\item the coordinate change $x_3\mapsto -x_3$ preserves $t_1,t_2$ and maps $t_3\mapsto t_3+2^{m-1}$.
\end{itemize}
For two choices of weights $(t_1,t_2,t_3)$, if the corresponding characters $\chi$  differ by $\pm1$ and the induced $\fD_n$-actions on $\mathfrak l$ are isomorphic, then the weights can be interchanged by a composition of the three maps above, and thus the actions are birational. 
\end{proof}

\subsection{When $\bar G=\fD_n$ with even $n$}

The generic stabilizer $C_t$ of $\mathfrak l$ is never trivial,
and all possibilities for $G$ are listed in Proposition~\ref{prop:intrans}. 
The standard Cremona involution $\iota$ is contained in $\mathrm{Bir}^G(\mathbb{P}^2)$,
and it fits in the $G$-commutative diagram:
$$
\xymatrix{
&S\ar@{->}[dr]^{\phi}\ar@{->}[dl]_{\pi}&\\%
\bF_{1}\ar@{->}[d]_{\eta}&&\bF_{1}\ar@{->}[d]^{\eta}\\
\bP^2\ar@{-->}[rr]_{\iota}&&\bP^2}
$$
where $\eta$ is the blowup of the $G$-fixed point,
$\pi$ is the blowup of the orbit of length $2$ in the strict transform $L$ of the line $\mathfrak{l}\subset\bP^2$,
and $\phi$ is the contraction of the strict transforms of the fibers of the natural projection $\bF_1\to\bP^1$ that pass through this orbit.

Observe that $\mathrm{Bir}^G(\mathbb{P}^2)$ also contains the birational selfmap of $\mathbb{P}^2$ given by
\begin{equation}
\label{Dn-map}
\gamma: (x_1,x_2,x_3)\mapsto (x_1x_2^{n/2}x_3^{n/2},x_2(x_2^n-x_3^n),x_3(x_2^n-x_3^n)).
\end{equation}
To describe the geometry of $\gamma$, let $E$ be the $\eta$-exceptional curve.
Then $\gamma$ fits in the following $G$-commutative diagram: 
$$
\xymatrix{
&S_1\ar@{->}[dr]^{\phi_1}\ar@{->}[dl]_{\pi_1}&&S_2\ar@{->}[dr]^{\phi_2}\ar@{->}[dl]_{\pi_2}&&S_{r}\ar@{->}[dr]^{\phi_r}\ar@{->}[dl]_{\pi_r}&\\%
\bF_{1}\ar@{->}[d]_{\eta}&&\bF_{n-1}&&\bF_{n-2}\cdots \,\bF_3&&\bF_1\ar@{->}[d]^{\eta}\\
\bP^2\ar@{-->}[rrrrrr]_{\gamma}&&&&&&\bP^2}
$$
where $r=\frac{n}{2}+1$, $\pi_1$ is the blowup of the orbit of length $n$ in $L$,
$\phi_1$ is the contraction of the strict transforms of the fibers of the natural projection $\bF_1\to\bP^1$ that pass through this orbit,
and each $\pi_i$ for $i\geqslant 2$ is the blowup of the orbit of length $2$ on the strict transform of $E$,
and each $\phi_i$ for $i\geqslant 2$ is the contraction of the strict transforms of the fibers of the natural projection $\bF_{n+3-2i}\to\bP^1$ that pass through this orbit.

\begin{proof}[Proof of Theorem~\ref{thm:intransBurn} in this case]
In each of the families, the $G$-action depends on primitive weights $t_1,t_2,t_3,t_4$ as  in  
\begin{enumerate}
    \item $
 \mathrm{diag}(1,\zeta_r^{t_1},\zeta_r^{t_1}),\quad \mathrm{diag}(1,\zeta_{2n}^{t_2},\zeta_{2n}^{-t_2}),\quad  (x_1,\zeta_4x_3,\zeta_{4}x_2).
 $
 \item $
 \mathrm{diag}(1,\zeta_r^{t_1},\zeta_r^{t_1}),\quad \mathrm{diag}(1,\zeta_{2^{m+1}}^{t_3}\zeta_{2n}^{t_2},\zeta_{2^{m+1}}^{t_3}\zeta_{2n}^{-t_2}),\quad  (x_1,\zeta_4x_3,\zeta_{4}x_2).
 $
  \item $
 \mathrm{diag}(1,\zeta_r^{t_1},\zeta_r^{t_1}),\quad \mathrm{diag}(1,\zeta_{2n}^{t_2},\zeta_{2n}^{-t_2}),\quad  (x_1,\zeta_{2^{m+1}}^{t_4}\zeta_4x_3,\zeta_{2^{m+1}}^{t_4}\zeta_{4}x_2).
 $
\end{enumerate} 
The generic stabilizer of $\mathfrak l=\{x_1=0\}\subset\bP^2$ is generated by 
\begin{enumerate}
    \item $\mathrm{diag}(1,\zeta_r^{t_1},\zeta_r^{t_1}),\quad \mathrm{diag}(1,-1,-1)$;
    \item  $\mathrm{diag}(1,\zeta_r^{t_1},\zeta_r^{t_1}), \quad \mathrm{diag}(1,\zeta_{2^m}^{t_3},\zeta_{2^m}^{t_3})$;
     \item  $\mathrm{diag}(1,\zeta_r^{t_1},\zeta_r^{t_1})\quad \mathrm{diag}(1,\zeta_{2^m}^{t_4},\zeta_{2^m}^{t_4})$.
\end{enumerate}
    The weights can be altered by the following maps, in all cases
    \begin{itemize}
        \item the birational map with image $(\frac{1}{x_1},\frac{1}{x_3},-\frac{1}{x_2})$ maps
        $$
        t_1\mapsto -t_1,\quad t_2\mapsto t_2,\quad t_3\mapsto -t_3,\quad t_4\mapsto -t_4,
        $$
        \item the coordinate change $\mathrm{diag}(1,1,-1)$ maps
          $$
        t_1\mapsto t_1, \quad t_2\mapsto t_2,\quad t_3\mapsto t_3,\quad t_4\mapsto t_4+2^m.
        $$
         \item the birational map $\gamma$ maps 
       $$
         t_1\mapsto t_1, \quad t_2\mapsto t_2,\quad t_3\mapsto t_3+2^m,\quad t_4\mapsto t_4,
        $$
        or equivalently, 
         $$
        t_1\mapsto t_1, \quad t_2\mapsto t_2+n,\quad t_3\mapsto t_3,\quad t_4\mapsto t_4.
        $$
    \end{itemize}
    Similarly, for two actions, if the characters $\chi$ of the generic stabilizer of $\mathfrak l$ differ by $\pm1$ and the induced $\fD_n$-actions on $\mathfrak l$ are isomorphic, we know that the corresponding weights can be interchanged by a composition of the three maps above; thus, the actions are birational. 
\end{proof}

To obtain an alternative proof of Theorem~\ref{thm:intransBurn} in this case,
one can explicitly describe generators of  $\mathrm{Bir}^G(\mathbb{P}^2)$:
Recall that $\iota,\gamma\in\mathrm{Bir}^G(\mathbb{P}^2)$. 
For every $\lambda\in k^\times$, with $\lambda^n\ne 1$,
let $\tau_\lambda\in \mathrm{Bir}^G(\mathbb{P}^2)$ be the birational map given by
\begin{equation}
\label{Dn-map-sigma}
\tau_\lambda: (x_1,x_2,x_3)\mapsto (x_1(x_2^{2n}-x_3^{2n}),x_2f_{\lambda},x_3f_{\lambda}),
\end{equation}
where 
$$
f_{\lambda}=\prod_{r=0}^{n-1}\left((x_1-\lambda\zeta_n^r x_2)\cdot(x_2-\lambda\zeta_n^r x_1)\right).
$$
Then $\tau_\lambda$ fits in the following $G$-commutative diagram:
$$
\xymatrix{
&S^\prime\ar@{->}[dr]^{\phi^\prime}\ar@{->}[dl]_{\pi^\prime}&&S^{\prime\prime}\ar@{->}[dr]^{\phi^{\prime\prime}}\ar@{->}[dl]_{\pi^{\prime\prime}}&&S^{\prime\prime\prime}\ar@{->}[dr]^{\phi^{\prime\prime\prime}}\ar@{->}[dl]_{\pi^{\prime\prime\prime}}&\\%
\bF_{1}\ar@{->}[d]_{\eta}&&\bF_{2n-1}&&\bF_{n-1}&&\bF_1\ar@{->}[d]^{\eta}\\
\bP^2\ar@{-->}[rrrrrr]_{\tau_{\lambda}}&&&&&&\bP^2}
$$
where $\pi^\prime$ is the blowup of the orbit of length $2n$ on $L$ that is mapped by $\chi$ to the $G$-orbit of $[0:\lambda:1]$,
$\pi^{\prime\prime}$ and $\pi^{\prime\prime\prime}$ are blowups of the orbits of length $n$ on the strict transform of $E$,
and $\phi^\prime$, $\phi^{\prime\prime}$ and $\phi^{\prime\prime\prime}$ are contractions of the strict transforms of the fibers of the corresponding $\mathbb{P}^1$-bundles passing through the blown up orbits.

\begin{lemm} 
\label{lemma:Dn}
Suppose $n\geqslant 10$. Then  
$\Bir^G(\bP^2)$ is generated by $\Aut^G(\bP^2)$, the standard Cremona transformation $\iota$, the map  $\gamma$ given by \eqref{Dn-map}, and  $\tau_\lambda$ given by \eqref{Dn-map-sigma} for $\lambda\in k^\times$ with  $\lambda^n\ne 1$.
\end{lemm}

\begin{proof}
Take any $\varphi\in \Bir^G(\bP^2)$ such that $\varphi\not\in\Aut^G(\bP^2)$.
Then $\varphi$ can be decomposed into a sequence of $G$-Sarkisov links.
Since $n\geqslant 10$, $\mathbb{P}^2$ contains exactly two $G$-orbits of length $<9$: the $G$-fixed point, and the unique $G$-orbit of length $2$, which is contained in the $G$-invariant line.  Hence, the only $G$-Sarkisov links that start at $\mathbb{P}^2$ are the link 
$$
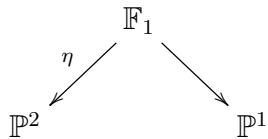
\xymatrix{
&\bF_{1}\ar@{->}[dl]_{\eta}\ar@{->}[dr]\\
\bP^2&&\bP^1}
$$
where $\bF_{1}\to\bP^1$ is the natural projection, and the link
$$
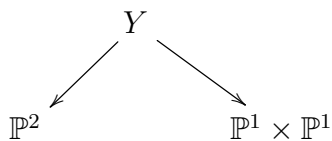
\xymatrix{
&Y\ar@{->}[dl]\ar@{->}[dr]\\
\bP^2&&\bP^1\times\bP^1}
$$
where the left map is the blow up of the unique $G$-orbit of length $2$
and the right map is the contraction of the strict transform of the $G$-invariant line $\mathfrak{l}\subset\mathbb{P}^2$.
Moreover, the $G$-fixed point in $\bP^1\times\bP^1$ is the only $G$-orbit of length $<8$ in $\bP^1\times\bP^1$,
so the only $G$-Sarkisov link that starts at $\bP^1\times\bP^1$ is the inverse of the link described above.

Moreover, since $C_t\ne 1$, any $G$-orbit of a point in $\mathbb{F}_1$ away from $E\cup L$ intersects some fiber of the natural projection $\mathbb{F}_1\to\mathbb{P}^1$ at more than 1 point. Now, using the classification of two-dimensional $G$-Sarkisov links \cite{Isk}, 
we see that we have the following $G$-commutative diagram 
$$
\xymatrix{
&\bF_{1}\ar@{-->}[rr]^{\rho}\ar@{->}[ddl]_{\eta}\ar@{->}[d]&&\bF_{1}\ar@{->}[d]\ar@{->}[ddr]^{\eta}&\\
&\mathbb{P}^1\ar@{->}[rr]_{\upsilon}&&\mathbb{P}^1&\\
\bP^2\ar@{-->}[rrrr]_{\varphi}&&&&\bP^2}
$$
where $\upsilon$ is an isomorphism, and  $\rho$ is a $G$-birational map such that its indeterminacy points are $G$-orbits in $E\cup L$, and every fiber of the projection $\bF_{1}\to\mathbb{P}^1$ contains at most one of these indeterminacy points. Composing $\varphi$ with a composition of appropriate maps $\tau_\lambda$ or $\iota\circ\tau_\lambda\circ\iota$, we may assume that the indeterminacy points of the map $\rho$ are $G$-orbits of length $2$ or $n$. Composing $\varphi$ with $\gamma$ or $\iota\circ\gamma\circ\iota$, we may further assume that the indeterminacy points of  $\rho$ are $G$-orbits of length $2$. Hence, either $\rho\in\Aut^G(\bP^2)$ or $\rho\circ\iota\in\Aut^G(\bP^2)$, which implies the required assertion.
\end{proof}    
    
It seems possible to weaken (or remove) the condition $n\geqslant 10$ in Lemma~\ref{lemma:Dn}; this would entail taking into account Geiser and Bertini involutions that may appear in $\Bir^G(\bP^2)$. 

\subsection{When $\bar G=\fA_4$}
Recall that there is a unique $\fA_4$-action on $\bP^1$. We find birational maps between two $G$-actions when the corresponding characters $\chi$ differ by $\pm1$. 
\begin{proof}[Proof of Theorem~\ref{thm:intransBurn} in this case]
    The action depends on choices of primitive (with the exception that $t_2$ can be 0) weights $t_1,t_2, t_3$ in 
    \begin{enumerate}
       \item $\mathrm{diag}(1,\zeta_r^{t_1},\zeta_r^{t_1}),\,\,(x_1,\zeta_{3}^{t_2}\zeta_3^2x_2,\zeta_{3}^{t_2}(\zeta_3x_3-x_2))$;
       \item $\mathrm{diag}(1,\zeta_r^{t_1},\zeta_r^{t_1}),\,\, (x_1,\zeta_{3^{m+1}}^{t_3}\zeta_{3}^{t_2}\zeta_3^2x_2,\zeta_{3^{m+1}}^{t_3}\zeta_{3}^{t_2}(\zeta_3x_3-x_2))
    $
    \end{enumerate}
    with $m\geq 1$. The respective action of the generic stabilizer of $\mathfrak l$ is generated by 
    \begin{enumerate}
    \item $    \mathrm{diag}(1,\zeta_r^{t_1},\zeta_r^{t_1}),\quad \mathrm{diag}(1,-1,-1);
    $
        \item 
$    \mathrm{diag}(1,\zeta_r^{t_1},\zeta_r^{t_1}),\quad \mathrm{diag}(1,\zeta_{3^{m}}^{t_3},\zeta_{3^{m}}^{t_3}),\quad \mathrm{diag}(1,-1,-1).
    $
        \end{enumerate}
Let 
\begin{align*}
    h_4&=x_2^4 - \frac{8\zeta_3+4}{3}x_2^3x_3 - 2x_2^2x_3^2 + \frac{8\zeta_3 + 
            4}{3}x_2x_3^3 + x_3^4,\\
     f_4&=x_2^3x_3 + (-2\zeta_3 - 1)x_2^2x_3^2 - x_2x_3^3,\\
     f_6&= x_2^6 - (4\zeta_3 + 2)x_2^5x_3 - 5x_2^4x_3^2 - 5x_2^2x_3^4 + 
            (4\zeta_3 + 2)x_2x_3^5 + x_3^6.
\end{align*}
We note that $f_4,h_4,g_6$ are semi-invariant forms under the $G$-action. They correspond to orbits of length 4 and 6 of the $\fA_4$-action on $\bP^1$. Then
\begin{itemize}
    \item The birational map $\sigma:(x_1,x_2,x_2)\mapsto (x_1h_4,x_2f_4,x_3f_4)$ maps 
    $$
    t_1\mapsto t_1,\quad t_2\mapsto t_2+1,\quad t_3\mapsto t_3,
    $$
    or equivalently 
    $$
    t_1\mapsto t_1,\quad t_2\mapsto t_2,\quad t_3\mapsto t_3+3^{m-1}.
    $$
    \item The birational map $\gamma: (x_1,x_2,x_3)\mapsto (f_6h_4,x_1x_2f_4^2,x_1x_3f_4^2)$ maps 
    $$
    t_1\mapsto -t_1,\quad t_2\mapsto -t_2,\quad t_3\mapsto-t_3.
    $$
\end{itemize}
 For two $G$-actions, as long as the characters of the stabilizer of $\mathfrak l$ differ by $\pm 1$, we can interchange the corresponding weights $t_1,t_2,t_3$ of the two actions using maps above, and thus they are birational. 
\end{proof}

\subsection{When $\bar G=\fS_4$}
Similarly, there is a unique $\fS_4$-action on $\bP^1$. We find birational maps between two actions when their corresponding characters $\chi$ differ by $\pm1$. 
\begin{proof}[Proof of Theorem~\ref{thm:intransBurn} in this case]
 The action depends on choices of primitive weights $t_1,t_2, t_3$ in 
$$
\mathrm{diag}(1,\zeta_r^{t_1},\zeta_r^{t_1}),\quad (x_1,\zeta_{2^{m+1}}^{t_2}\zeta_{4}^{t_3}(-2x_2+s_1x_3), \zeta_{2^{m+1}}^{t_2}\zeta_4^{t_3}(s_2x_2+2x_3))
$$
with $m\geq 0$.    The corresponding action of the generic stabilizer of $\mathfrak l$ is generated by 
    \begin{itemize}
    \item $    \mathrm{diag}(1,\zeta_r^{t_1},\zeta_r^{t_1}),\quad \mathrm{diag}(1,-1,-1),\quad 
    $ when $m=1$,
        \item 
$    \mathrm{diag}(1,\zeta_r^{t_1},\zeta_r^{t_1}),\quad\mathrm{diag}(1,-\zeta_{2^{m}}^{t_2},-\zeta_{2^{m}}^{t_2}), \quad
    $ when $m\ne 1$.
        \end{itemize}
        Let 
        $f_6,f_8,f_{12},h_{12}$ be semi-invariant binary forms of degree $6$, $8$, 
    $12$ and $12$, under the $\fS_4$-action on $\bP^1_{x_2,x_3}$ generated by 
        $$
        (-2x_2+s_1x_3,s_2x_2+2x_3),\quad (-s_2x_2-x_3,s_1x_2+(s_1+1)x_3)
        $$
        such that $f_{12}$ and $h_{12}$ are linearly independent. Such forms are unique up to scalars. Then 
\begin{itemize}
    \item The birational map $\sigma:(x_1,x_2,x_2)\mapsto (x_1h_{12},x_2f_{12},x_3f_{12})$ maps 
    $$
    t_1\mapsto t_1,\quad t_2\mapsto t_2+2^m,\quad t_3\mapsto t_3,
    $$
    or equivalently 
    $$
    t_1\mapsto t_1,\quad t_2\mapsto t_2,\quad t_3\mapsto t_3+2.
    $$
    \item The birational map $\gamma: (x_1,x_2,x_3)\mapsto (f_8,x_1x_2f_6,x_1x_3f_6)$ maps 
    $$
    t_1\mapsto -t_1,\quad t_2\mapsto -t_2,\quad t_3\mapsto-t_3.
    $$
\end{itemize}
When the characters of the stabilizer of $\mathfrak l$ in two actions differ by $\pm 1$, we can interchange the corresponding weights $t_1,t_2,t_3$ using the maps above, and thus the actions are birational.
\end{proof}

\subsection{When $\bar G=\fA_5$}

Recall that the group $G$ is generated by 
$$
\mathrm{diag}(1,\zeta_r,\zeta_r)
$$
and 
\begin{align}\label{eqn:sl25}
(x_1,(\zeta_{5}^3+\zeta_5^4)x_2-(\zeta_5^4+1)x_3,\zeta_5^3x_2+(\zeta_5^2+\zeta_5)x_3),(x_1,\zeta_5^2x_3,-\zeta_5^3x_2).
\end{align}
In particular, there is a unique minimal lift of $\fA_5\subset\PGL_2$ to $\GL_2$ given by \eqref{eqn:sl25}. This minimal lift is isomorphic to $\SL_2(\bF_5)$. It has two faithful 2-dimensional representations, giving rise to $2$ non-isomorphic $\fA_5$-actions on $\bP^1$.

\begin{proof}[Proof of Theorem~\ref{thm:intransBurn} in this case]  Fixing the  $\fA_5$-action on $\bP^1$, we see that the isomorphism class of the $G$-action only depends on the weight of the generic stabilizer of $\mathfrak l$. The action of the stabilizer is generated by 
 \begin{itemize}
     \item $
 \mathrm{diag}(1,\zeta_r^{t_1},\zeta_r^{t_1}),\quad 
 $ if $r$ is even,
 \item $
 \mathrm{diag}(1,\zeta_r^{t_1},\zeta_r^{t_1}),\quad  \mathrm{diag}(1,-1,-1),\quad 
 $ if $r$ is odd,
 \end{itemize}
 where $t_1$ indicates the weight. Let $f_{12}$, $f_{20}$ and $f_{30}$ be the $\SL_2(\bF_5)$-invariant binary forms of degrees $12$, $20$ and $30$ in variables $x_2$ and $x_3$, which correspond to the unique orbits of lengths 12, 20 and 30 of the $\fA_5$-action on $\bP^1$. Then the birational map 
 $$
 \sigma: (x_1,x_2,x_3)\mapsto (f_{12}f_{20},x_1x_2f_{30},x_1x_3f_{30})
 $$
changes the weight $t_1\mapsto -t_1$. We conclude that two actions are birational if and only if the characters of the stabilizer of $\mathfrak l$ differ by $\pm1$.
\end{proof}

This completes the proof of Theorem~\ref{thm:intransBurn}, and Theorem~\ref{thm:intran} follows as a corollary.

\bibliographystyle{alpha}
\bibliography{linear-p2}

\end{document}